\documentclass[11pt]{amsart} 

\usepackage{amssymb}
\usepackage[all]{xy}
\swapnumbers

\theoremstyle{plain}
\newtheorem{theorem}{Theorem}[section]
\newtheorem*{theorem*}{Theorem}
\newtheorem{definition}[theorem]{Definition}
\newtheorem{lemma}[theorem]{Lemma}
\newtheorem{corollary}[theorem]{Corollary}
\newtheorem{proposition}[theorem]{Proposition}
\newtheorem{fact}[theorem]{Fact}

\theoremstyle{definition}
\newtheorem{remark}[theorem]{Remark}
\newtheorem{example}[theorem]{Example}
\newtheorem{examples}[theorem]{Examples}

\newtheorem*{claim*}{Claim}
\newtheorem*{definition*}{Definition}

\newcommand{\C}{\mathcal C}
\renewcommand{\H}{\mathcal H}
\newcommand{\K}{\mathbb K}
\newcommand{\U}{\mathcal U}
\newcommand{\V}{\mathcal V}

\newcommand{\reel}{\mathbb R}

\newcommand{\To}{\mathop{\longrightarrow}}
\newcommand{\clspan}{\overline{\rm span\,}}
\newcommand{\Span}{\mathrm{span}\,}

\newcommand{\bbar}{\overline}

\begin{document}
\title[Asymptotically hilbertian spaces and uncountable categoricity]{\rm Asymptotically hilbertian\\ modular Banach spaces:\\examples of uncountable categoricity}

\author[C.\ W.\ Henson]{C.\ Ward Henson*}

\address{C. Ward \textsc{Henson} \\
  University of Illinois at Urbana-Champaign \\
  Urbana, Illinois 61801 \\
  USA}

\author[Y.\ Raynaud]{Yves Raynaud}

\address{Yves Raynaud, Institut de Math\'ematiques de Jussieu-Paris Rive Gauche, CNRS/UPMC (Univ.Paris 06)/Univ. Paris-Diderot, 4 place Jussieu, F-75252 Paris Cedex 05, France}

\thanks{(*) Research of this author was supported by grant \#202251 from the Simons Foundation.\\
Part of the work was carried out during the Universality and Homogeneity program at the Hausdorff Institute of the University of Bonn (Fall 2013).}
\maketitle
\rightline{}

\begin{abstract}

 We give a criterion ensuring that the elementary class of a modular Banach space $E$ 
(that is, the class of Banach spaces, some ultrapower of which is linearly isometric to an ultrapower of $E$)  consists of all direct sums $E\oplus_m H$, where $H$ is an arbitrary Hilbert space and $\oplus_m$ denotes the modular direct sum.  Also, we give several families of examples in the class of Nakano direct sums of finite dimensional normed spaces that satisfy this criterion. This yields many new examples of uncountably categorical Banach spaces, in the model theory of Banach space structures.
\end{abstract}
\rightline{}
\section{Introduction}

The aim of this paper is to give some new examples of uncountably categorical Banach space structures. The motivation is model-theoretic, but here we formulate our objectives and methods of proof in the framework of ordinary Banach space theory, using the well known ultrapower construction \cite{H}. 

To begin, we give some terminology and explain briefly the model-theoretic background.  Two Banach spaces  $X$, $Y$ are {\it elementarily equivalent} if for some ultrafilters $\U$ and $\V$, their respective ultrapowers $X_\U$ and $Y_\V$ are linearly isometric. This is an equivalence relation, although its transitivity is not evident. In fact this relation is identical to that of approximate elementary equivalence in the first author's logic for normed space structures \cite[Discussion p.\ 5 and Theorem 10.7]{HI} and also to that of elementary equivalence in continuous logic applied to the unit balls \cite[Definition 4.3, Corollary 5.6 and Theorem 5.7]{BBHU}. These logical counterparts are defined to mean that the two Banach spaces satisfy the same sentences having the appropriate syntactic form, and this makes it clear that the elementary equivalence relation defined using ultrapowers is transitive. The  class of Banach spaces that are elementarily equivalent to a given Banach space is called the \emph{elementary class} of this Banach space; it consists of all ultraroots of ultrapowers of the given space. (An {\it ultraroot} of a Banach space $X$ is a Banach space $Y$, an ultrapower of which is linearly isometric to $X$).

It follows from basic results of model theory, the Compactness Theorem as well as the L\"owenheim-Skolem Theorems (in either of the two logics mentioned above), that the elementary class of any infinite dimensional Banach space contains spaces of all infinite density characters. If $\kappa$ is an infinite cardinal number, an infinite dimensional Banach space $X$ is said to be {\it $\kappa$-categorical} if its elementary class contains exactly one member of density character $\kappa$. It is called {\it uncountably categorical} if it is $\kappa$-categorical for some uncountable cardinal $\kappa$; in that case it is $\kappa$-categorical for all uncountable cardinals $\kappa$, as proved independently in \cite{BY-08} and \cite{SU-11}. 

Note that uncountable categoricity is a property of the elementary class of $X$ rather than of $X$ by itself, so we should speak of an uncountably categorical class of Banach spaces.  So we say that a class $\C$ of Banach spaces is uncountably categorical if and only if $\C$ is closed under ultrapowers and under ultraroots (hence it is closed under elementary equivalence), and it contains only one member $X_\kappa$ of density character $\kappa$ for some uncountable cardinal $\kappa$; in that case all infinite dimensional members of $\C$ will necessarily be elementarily equivalent to $X_\kappa$.

A trivial example  is the class of Hilbert spaces. Indeed, up to linear isometry there is only one infinite dimensional Hilbert space of any given density character; moreover, ultrapowers of any Hilbert space are Hilbert spaces, and so are closed subspaces (hence also ultraroots).  

Although the structure of uncountably categorical Banach space structures has begun to be investigated \cite{SU-14}, very few examples are known. In addition to the class of Hilbert spaces, some less trivial examples consist of certain finite dimensional perturbations of Hilbert spaces: namely given a finite dimensional normed space $E$, the class $\mathcal C_E$ of all 2-direct sums $E\oplus_2 H$, where $H$ is any infinite dimensional Hilbert space, is $\kappa$-categorical for all infinite cardinal $\kappa$. This fact was known to the authors for a few years; its proof is included here (in Section \ref{sec:E+H}) for the sake of completeness and because it is closely tied to the other results that are expounded here. The main purpose of this article is, however, to give some less trivial examples, which include many spaces that are not linear-topologically isomorphic to a Hilbert space; one of them is even not linear-topologically embeddable in any Banach lattice with nontrivial concavity.  So from the point of view of Banach space geometry, they are not close to the class of Hilbert spaces.  On the other hand, in a certain asymptotic sense, which will be explained in Section 3, they have a very hilbertian character.  (In a sense that is not yet well understood, the main result of \cite{SU-14} says that an uncountably categorical Banach space must be very closely related to the class of Hilbert spaces.)  

Some other model-theoretic properties of our examples will be discussed in a future paper.  Here we concentrate on the methods needed to prove their categoricity,  using purely Banach space theoretic arguments which have their own interest, in particular in modular sequence space theory.

This paper is organized as follows: in Section \ref{sec:General} we describe a general pattern of the examples that we present, in the framework of {\it modular Banach spaces}. We give a criterion guaranteeing that the elementary class of a (separable, infinite dimensional) modular space $E$, considered as a Banach space with no extra structure, consists exactly of all the modular direct sums $E\oplus_m H$, where $H$ is a Hilbert space of arbitrary hilbertian dimension; this property clearly implies uncountable categoricity. All of our examples fit into this framework.  (The concepts of modular Banach space and modular direct sum are explained in Section \ref{sec:General}.  We note that every Banach space can be considered as a modular space by taking its modular function to be the square of the norm.)  We then make a brief presentation of the concrete examples whose properties are demonstrated later in the paper. All these examples are given as Nakano direct sums of a sequence of finite dimensional spaces. The Nakano (sequence) spaces appearing there are associated to a sequence of exponents converging to 2.  In Section \ref{sec:asympt-H-Nakano}, ultrapowers of these Nakano direct sums are described; it is proved that when the geometry of the sequence of finite dimensional spaces converges to the geometry of Hilbert space in a very weak sense, then the Nakano direct sum is 
\emph{asymptotically hilbertian}. 

The main content of this paper is in Sections \ref{sec:Nakano-direct-sums} and \ref{sec:2-direct-sums}.  There are two main families of examples discussed there: in the first one, which is treated in Section \ref{sec:Nakano-direct-sums}, all exponents of the underlying Nakano sequence space are distinct from 2; in the second one, treated in Section \ref{sec:2-direct-sums}, all the exponents are equal to 2. In both cases, the sequence of finite dimensional spaces used in the direct sum must converge in a suitable local sense to Hilbert space. There is more flexibility in the case of exponents different from 2: here the factor spaces may be even one-dimensional; in particular, the Nakano space itself is an example.  This space is not linear-topologically equivalent to Hilbert space if the convergence of the exponents to 2 is slow enough.

We finish with section \ref{sec:E+H}, which treats the finite dimensional perturbations of the class of Hilbert spaces already mentioned above, as an addendum to section \ref{sec:2-direct-sums}.

Throughout this paper, the Banach spaces considered may be either real or complex, and our proofs make no distinction between real and complex scalars.  We denote the scalar field by $\K$.  For general notions about ultrapowers the reader is referred to \cite[Section I]{H}.  (Our notation differs slightly from that of this author; \textit{e.g.}, we use $E_\U$, $[x_i]_\U$,  in place of $(E)_\U$, $(x_i)_\U$, etc).

\section{A general pattern}\label{sec:General}

The Banach space examples $E$ presented in this paper have the common feature that every ultrapower of $E$ is linearly isometric to a direct sum $E\oplus H$, where $H$ is a Hilbert space.  (It can be shown that this condition implies that $E$ is reflexive; we omit the argument since we have easier direct proofs of reflexivity for the examples presented here.)

Indeed, if $E$ is one of our examples, we show that for any ultrafilter $\U$ we have a direct sum decomposition 
\begin{equation}\label{eq:decomposition}
E_\U= D(E)\oplus H
\end{equation} 
where $D \colon E\to E_\U$ is the canonical embedding of $E$ into its ultrapower $E_\U$ (namely, $D$ assigns to each $x\in E$ the element represented in $E_\U:=E^I/\U$ by the constant family $(x)_{i\in I}$) and where $H$ is a closed linear subspace of $E_\U$ that is linearly isometric to a Hilbert space. 

Then $D(E)$ is complemented in $E_\U$ by the weak limit projection $P: [x_i]_\U\mapsto w\hbox{-}\lim\limits_{i,\U} x_i$ (which exists since $E$ is reflexive).  If $\ker P$ is linearly isometric to a Hilbert space (\textit{i.e.}, equation (\ref{eq:decomposition}) holds with $H=\ker P$) we say that $E$ is {\it asymptotically hilbertian (in the isometric sense)}. Indeed this is an isometric version of the notion of asymptotically hilbertian space considered, \textit{e.g.}, in \cite[pp. 220-221]{P89}, which corresponds to the case where $\ker P$ is only linear-topologically isomorphic to a Hilbert space. The examples that we present here are indeed asymptotically hilbertian (in the isometric sense), but the reader will observe that this property is not formally required by the main theoretical tool of the present section (Theorem \ref{Criterion:UC}).

To go further toward a similar description of all members of the elementary class of $E$, we need to make hypotheses on the nature of the direct sum. For example, if the direct sum in equation (\ref{eq:decomposition}) is always a 2-sum, that is, if
\[\|x+h\|^2=\|x\|_E^2+\|h\|_H^2\]
holds for all $x\in E$, $h\in H$, then for any Hilbert space $K$:
\[ (E\oplus_2 K)_\U=E_\U\oplus_2 K_\U= (E\oplus_2 H)\oplus_2 K_\U= E\oplus_2(H\oplus_2 K_\U)= E\oplus_2 K'\]
where $K'=H\oplus_2 K_\U$ is still a Hilbert space.  It follows that the class $\C$ of direct sums $E\oplus_2 K$, where $K$ is a Hilbert space, is closed under ultrapowers. Moreover, if $E$ is separable, then for every uncountable cardinal $\kappa$, the only member of $\mathcal C$ with density character $\kappa$ is $E\oplus_2 H_\kappa$, where $H_\kappa$ is the Hilbert space of hilbertian dimension $\kappa$. Therefore, if $\C$ turns out to be closed under ultraroots, it is necessarily uncountably categorical. 

Furthermore, there are other simple kinds of direct sums, apart from the 2-sums, for which the preceding reasoning is valid, namely the {\it modular} direct sums. A \emph{convex modular} on a linear space $X$ is a convex function $\Theta: E\to \mathbb R_+$ that satisfies the following conditions: $\Theta(0)=0$, $\Theta$ is symmetric ($\Theta(\lambda x)=\Theta(x)$ for any scalar $\lambda$ with $|\lambda|=1$), and $\Theta$ is faithful ($\Theta(x)=0\implies x=0$).  An associated norm on $X$ is then defined by the Luxemburg formula
\begin{equation}\label{eq:Lux-norm}
\|x\|=  \inf\{\lambda>0: \Theta(x/\lambda)\le 1\} \text{.}
\end{equation}
Since we require that the convex modular has finite values, the Luxemburg norm is implicitly defined by the equation
\begin{equation}\label{eq:Lux-norm-1}
\Theta\left(x\over \|x\|\right)=1 \text{.}
\end{equation}
The modular direct sum $X_1\oplus_m X_2$ of two modular spaces $(X_1,\Theta_1)$, $(X_2,\Theta_2)$ is their linear topological direct sum equipped with the modular
\[ \Theta(x_1+x_2)=\Theta_1(x_1)+\Theta_2(x_2)\quad x_1\in X_1, x_2\in X_2 \,\text{.}\]
Note that if norms $\|x\|_{X_i}$ on $X_i$ are given (for $i=1,2$), then convex modulars can be defined by $\Theta_i(x)=\|x\|_{X_i}^2$; moreover, the  Luxemburg norms associated to these modulars coincide with the given norms and the $m$-direct sum coincides with the $2$-direct sum. We shall systematically equip Hilbert spaces with their trivial modular $\Theta_H(x)=\|x\|^2$.

We say that a modular space $(X,\Theta)$ has a modular direct decomposition $X=Y\oplus_m Z$ if $Y$, $Z$ are linear subspaces of $X$ and for every $y\in Y,z\in Z$ we have $\Theta(y+z)=\Theta(y)+\Theta(z)$.

Given two modular spaces $(X_1,\Theta_1)$ and $(X_2,\Theta_2)$, we say that a linear map $T:X_1\to X_2$ is {\it modular preserving}, or {\it preserves modulars}, if $\Theta_2\circ T=\Theta_1$. Note that such a $T$ is necessarily isometric (for Luxemburg norms). Two modular spaces $(X_1,\Theta_1)$ and $(X_2,\Theta_2)$ are said to be {\it linearly isomodular} if there exists a modular preserving  surjective linear map from $X_1$ onto $X_2$. This implies that $X_1$ and $X_2$ are isometric, and the converse is trivially true in the special case where $\Theta_i(x)=\|x\|^2$, $i=1,2$. 

We say that a modular $\Theta$ satisfies the \emph{$\Delta_2$ condition with constant $C$} if we have 
\[
\Theta(2x)\le C \cdot \Theta(x) \ \ \ \ \ \mbox{for all $x\in X$.} 
\]
In this case the modular is bounded by $C^n$ on the ball $B(0, 2^n)$, and thus, by the general theory of convex functions, it is Lipschitz of constant at most $C^n$ on each ball $B(0, 2^n)$. One may thus define unambiguously a function $\Theta_\U$ on each ultrapower $X_\U$ by
\begin{align}\label{ultra-modular}
\Theta_\U([x_i]_\U)=[\Theta(x_i)]_\U \,\text{.}
\end{align}
It is immediate that $\Theta_\U$ is convex and symmetric. It is faithful since $\Theta_\U([x_i]_\U)=0$ means $\Theta(x_i)\to_{i,\U}0$ which implies $x_i\to_{i,\U} 0$ by equivalence of modular and norm convergence under the $\Delta_2$ condition. It is also easy to see using (\ref{eq:Lux-norm-1}) that the Luxemburg norm associated to $\Theta_\U$ coincides with the ultrapower norm on $X_\U$.

Our next result provides the main theoretical tool that we use in proving uncountable categoricity for the specific examples treated in the rest of the paper.

\begin{theorem}\label{Criterion:UC}
Let $E$ be an infinite dimensional $\Delta_2$ modular Banach space such that: 
\begin{itemize}
\item[\it i)] Every ultrapower of $E$ is linearly isomodular to $E\oplus_m H$ for some Hilbert space $H$.
\item[\it ii)] Every linear isometric operator $E\to E\oplus_m H$, where $H$ is any Hilbert space, maps $E$ onto $E$.
\end{itemize}
then the elementary class of $E$ consists exactly of all modular direct sums $E\oplus_m H$, where $H$ is an arbitrary Hilbert space, and hence this class is uncountably categorical (and $E$ is separable).

Furthermore, if $E_0$ is a finite dimensional $\Delta_2$ modular Banach space such that 
\begin{itemize}
\item[\it ii')] Any linear isometric operator $E_0\oplus_m \ell_2 \to E_0\oplus_m H$, where $H$ is any infinite dimensional Hilbert space, maps $E_0$ onto $E_0$
\end{itemize}
then the elementary class of $E=E_0\oplus_m \ell_2$ consists exactly of all modular direct sums $E_0\oplus_m H$, where $H$ is an arbitrary infinite dimensional Hilbert space, and hence this class is $\kappa$-categorical for every infinite cardinal $\kappa$.
\end{theorem}

\begin{remark}
 Condition $(i)$ implies that every ultrapower $E_\U$ has a modular direct sum decomposition $E_\U=E_1\oplus_m H$, where $E_1\subset E_\U$ is isomodular with $E$ and $H$ with an Hilbert space. Condition $(ii)$ applied to the diagonal embedding $D: E\to E_\U$ implies that $E_1=D(E)$, thus $E$ verifies eq. (\ref{eq:decomposition}).  However $H$ has no clear reason to be the kernel of the weak limit projection, and so $E$ may not be asymptotially hilbertian.
\end{remark}
\begin{proof}[Proof of Theorem \ref{Criterion:UC}]
Assume that $E$ is infinite dimensional and let $\mathcal C$ be the class of all Banach spaces that are linearly isometric to $E\oplus_m K$, for some Hilbert space $K$. This class is closed under ultrapowers since 
\[ (E\oplus_m K)_\U= E_\U\oplus_m K_\U \hbox{ is isomodular with } (E\oplus_m H)\oplus_m K_\U=E\oplus_m K'\]
where $K'=\H\oplus_m H_\U$ is also a Hilbert space. On the other hand, for every uncountable cardinal $\kappa$ strictly bigger than the density character of $E$, the only member (up to linear isometry) of the class $\mathcal C$ having density character $\kappa$ is $E\oplus_m H_\kappa$, where $H_\kappa$ is the Hilbert space of hilbertian dimension $\kappa$. To complete the proof of the Theorem it remains only to prove that the class $\mathcal C$ is also closed under ultraroots. Once this is proved, it follows that $E$ must be separable, since the elementary class of $E$ contains a separable space.

Let $X$ be  a Banach space such that some ultrapower $X_\U$ of $X$ is linearly isometric to a space $E\oplus_m H$. Let $J \colon  E \oplus_m H\rightarrow X_\U$ be such an isometry, $D_X \colon  X\hookrightarrow X_\U$ be the diagonal embedding and $i=J^{-1}D_X$ be the resulting embedding of $X$ into $E\oplus_m H$. Taking ultrapowers, we get an embedding $i_\U \colon  X_\U\hookrightarrow (E\oplus_mH)_\U$. By the preceding argument we have a modular direct decomposition $(E\oplus_m H)_\U= E_1\oplus_m K$, where $E_1$ is isometric (in fact, isomodular) to $E$ and $K$ is an Hilbert space. Let $D \colon E\oplus_m H\hookrightarrow (E\oplus_m H)_\U$ be the diagonal embedding. Using assumption (ii) we have that $D|E$ is an isometry from $E$ onto $E_1$, in particular $E_1=D(E)$. We set $S=i_\U J$; this is an isometric embedding $E\oplus_m H\hookrightarrow (E\oplus_m H)_\U$ which need not coincide with the diagonal embedding $D$. However, using assumption (ii) again, we do have that $S|E$ is an isometry from $E$ onto $D(E)$. To summarize, we have the following commutative diagram of linear isometries:

\[\xymatrix{&X_\U\ar^{i_\U}[dr]&\\X\ar^{D_X}[ur]\ar_i[r]\ar_i[rd]& E\oplus_m H \ar_J[u]\ar_S[r]&(E\oplus_m H)_\U&\hskip-1cm= D(E)\oplus_m K\\&E\oplus_m H\ar_{D}[ru]&}\] 

\medskip
Note that $i_\U(X_\U)\cap D(E)\subset Di(X)$. Indeed if $\xi=[x_i]_\U \in X_\U$ and $z\in E$ are such that  $i_\U([x_i]_\U)=[i(x_i)]_\U=D(z)$ then $i(x_i)\To\limits_\U z$.  Since $i(X)$ is a closed subspace of $E\oplus_m H$ this implies that $z\in i(X)$.

For every $x\in E$ we have $i_\U J(x)=S(x)\in D(E)$; hence by the preceding argument there is $x'\in X$ such that $S(x)=Di(x')$. As $S|E: E\to D(E)$ is surjective this shows that $D(E)\subset Di(X)$ and thus $E\subset i(X)$.

Let $\pi \colon  E\oplus_m H\to E$ be the first projection on the direct sum $E\oplus_m H$. By the preceding argument, the range of $\pi$ is contained in $i(X)$.  Let $\pi'$ be the restriction of $\pi$ to $i(X)$; since $E\subset i(X)$, the range of $\pi'$ is $E$. Its kernel $H_0$ is contained in that of $\pi$, and thus $H_0\subset H$ is a Hilbert space. Finally $i(X)=E\oplus H_0$, with the norm induced by that of $E\oplus_m H$. That is, $i(X)=E\oplus_m H_0$. 

Now suppose $E_0$ is a finite dimensional $\Delta_2$ modular Banach space that satisfies (ii).  Let $\mathcal C$ be the class of all Banach spaces that are linearly isometric to $E_0\oplus_m K$, for some Hilbert space $K$. Since $(E_0)_\U = E_0$ for any ultrafilter $\U$, the reasoning above shows that $\mathcal C$ is closed under ultrapowers and ultraroots.  
However, the members of this class are not all mutually elementarily equivalent, since any finite dimensional member $X$ of $\mathcal C$ has trivial ultrapowers ($X_\U=X$), and thus cannot be elementarily equivalent to an infinite dimensional one. The subclass $\mathcal C_\infty$ consisting of all the infinite dimensional members of $\mathcal C$ is also closed under ultrapowers and ultraroots, and has a unique member of any density character (up to linear isometry). Thus $\C_\infty $ is the elementary class of its unique separable member $E=E_0\oplus_m \ell_2$.   

It is routine to verify that the previous argument  for proving closedness by ultraroots also works for $\C_\infty$ under hypothesis (ii'), which is formally weaker than (ii).
\end{proof}

We  now introduce the examples that are treated in the next sections and summarize their connections to the hypotheses of the preceding Theorem; proofs of what we state here are given in Sections \ref{sec:asympt-H-Nakano}, \ref{sec:Nakano-direct-sums}, and \ref{sec:2-direct-sums}.

Typical examples of modular spaces satisfying condition (i)  of  Theorem \ref{Criterion:UC} are the Nakano sequence spaces $\ell_{(p_n)}$ associated to a sequence of exponents $(p_n)$ converging to 2. The space $N:=\ell_{(p_n)}$ is the linear space of sequences of (real or complex) scalars $x=(x(n))$ such that  
\[ \Theta(x):=\sum_{n=1}^\infty |x(n)|^{p_n}<\infty\]
and $\Theta$ is a natural convex modular on $N$ which verifies condition $\Delta_2$ (since the sequence $(p_n)$ is bounded). The Luxemburg norm on $N$ is then defined by (\ref{eq:Lux-norm}), and $N$ is complete for this norm. Since $|x|\le|y|$ clearly implies $\Theta(x)\le\Theta(y)$, it implies also $\|x\|\le\|y\|$; \textit{i.e.}, $N$ is a Banach lattice (for the Luxemburg norm). It is easy to see that if $(x_n)$ is a decreasing sequence of elements of $N$ which converges to zero coordinatewise then $\Theta(x_n)$ converges to zero and so does $\|x_n\|$; hence $N$ is order-continuous.

If $p_n\ne 2$ for all $n$, then $N$  also satisfies condition (ii) of Theorem \ref{Criterion:UC} (if $p_n$ is allowed to equal 2, this condition will only be true for the subspace $N_0=\clspan\,[e_n: p_n\ne 2]$). The space $N$ is linear-topologically isomorphic to $\ell_2$ if a certain summation condition due to Nakano holds (see Fact \ref{cond:N=ell 2} below). On the other hand, it is possible to choose the exponents to yield Nakano spaces that satisfy both conditions of Theorem \ref{Criterion:UC} but are not linear-topologically isomorphic to a Hilbert space, and they provide new kinds of examples of uncountable categoricity. 

A bigger variety of examples appears when we consider vector-valued Nakano sequence spaces, that is, direct sums of a sequence of finite dimensional Banach spaces. The elements of this direct sum are sequences of vectors, the norms of which form a sequence belonging to a given Nakano space. We denote by $(\mathop{\oplus}\limits_n E_n)_N$ the Nakano direct sum associated to the family of spaces $(E_n)$ and the Nakano space $N$. Thus
\[\big(\mathop{\oplus}\limits_n E_n\big)_N=\{(x(n))\in \prod_n E_n: (\|x(n)\|_{E_n})\in N\} \text{.}\]
A convex modular and norm are defined on the vector-valued Nakano sequence space by taking the Nakano modular of the sequence of norms:
\[ \Theta(x):=\sum_{n=1}^\infty \|x(n)\|_{E_n}^{p_n}\]
and then using the associated Luxemburg norm. We show that the Nakano direct sum  $(\mathop{\oplus}\limits_n E_n)_N$ satisfies condition (i) of  Theorem \ref{Criterion:UC} whenever the Jordan-von-Neumann constants (defined below, at the beginning of Section \ref{sec:asympt-H-Nakano}) of the spaces $E_n$ converge to 1 (here $p_n$ may take the value 2). These constants measure the degree of approximation to which the spaces satisfy the parallelogram inequality. It is equivalent to require that the Banach-Mazur distances from 2-dimensional subspaces of $E_n$ to the 2-dimensional Hilbert space converge uniformly to 1. (Note that this is a far weaker condition than saying that the Banach-Mazur distances from $E_n$ to the Hilbert space of the same dimension converge to 1.)

On the other hand, provided $p_n\ne 2$ for all $n$, the Nakano direct sum  $(\mathop{\oplus}\limits_n E_n)_N$ satisfies condition (ii) of  Theorem \ref{Criterion:UC} without any condition on the spaces $E_n$ except that they are finite dimensional. The reason is that isometries distinguish the spaces $E_n$ (or $H$) by the value of the corresponding exponent $p_n$ (resp. 2 for $H$). 

In the opposite case where $p_n$ is constantly $2$, in which case $N=\ell_2$ and the $N$-direct sum is a 2-sum, the preceding argument does not work, and isometries recognize the $E_n$ spaces rather by their geometric properties. For this reason the conditions on the spaces $E_n$  that we assume in this case are far more restrictive than in the preceding case (the examples are essentially the $\ell^p_n$ spaces or their non-commutative analogues, the Schatten classes $S^p_n$).

\section{Asymptotically hilbertian Nakano direct sums}\label{sec:asympt-H-Nakano}

\noindent For a normed space $X$, its Jordan-von Neumann constant $a(X)$ is defined by
\begin{align}\label{def:a(X)}
a(X)=\frac 12 \sup \{ \|x+y\|^2+\|x-y\|^2: x,y\in X, \|x\|^2+\|y\|^2=1\} \text{.}
\end{align}
By setting $u=x+y$, $v=x-y$ it is immediate that we also have
\begin{align}\label{eq:a(X)}
a(X)=2 \sup \{\|x\|^2+\|y\|^2: x,y\in X,  \|x+y\|^2+\|x-y\|^2 =1\} \text{.}
\end{align}
It follows that $a(X)\ge 1$, and that $a(X)=1$ iff $X$ is linearly isometric to a Hilbert space \cite{JvN}.
Note that $a(X)$ is the norm of the operator 
\[M_X: \ell_2^2(X)\to \ell_2^2(X): (x,y)\mapsto \frac 1{\sqrt 2}(x+y,x-y) \,\text{.}\]
The conjugate operator is easily seen to be $M_{X^*}$, so that $a(X^*)=a(X)$.

\smallskip
Let $N$ be a Nakano sequence space with exponent sequence $(p_n)$ converging to 2.

\noindent Let $(e_n)$ be the sequence of units of $N$ ($e_n$ is the sequence $(\delta_{kn})_{k\in\mathbb N}$). For every $n\in\mathbb N$ and $x\in N$ let $P_n(x)=\sum_{k=1}^n x(k)e_k$.  Clearly $P_n$ is a projection of norm one on $N$. On the other hand $\Theta(x-P_n(x))=\sum_{k=n+1}^\infty |x(k)|^{p_k}\to 0$ when $n\to\infty$, which by the $\Delta_2$-condition for $\Theta$ implies that $\|x-P_n(x))\|\to 0$. It follows that $(e_n)$ is a Schauder basis for $N$ and the generic element of $N$ can be written $x=\sum_{n=1}^\infty x(n)e_n$. (This basis is clearly unconditional with constant 1; in fact, it consists of mutually disjoint atoms of the Banach lattice $N$). 

Let $E=\left( \mathop{\oplus}\limits_n E_n\right)_N$ be the $N$-direct sum of a family $(E_n)$ of Banach spaces and $\Theta$ its modular, as defined in section \ref{sec:General}.
Note that since $(p_n)$ is bounded, the modular $\Theta$ satisfies a $\Delta_2$ condition.
If we set $\nu(x)=(\|x(n)\|_{E_n})$ we have clearly $\Theta_E(x)=\Theta_N(\nu(x))$ and $\|x\|_E=\|\nu(x)\|_N$. The map $\nu: X\to N$ is clearly 1-Lipschitz.

If $(e_n)$ is the natural basis of the Nakano space (as above),  it can be useful to denote the generic element $x=(x_n)$ of $E$ by $\sum_n e_n\otimes x(n)$. Setting $P_n(x)= \sum_{k=1}^n e_k\otimes x(k)$, the map $P_n$ is a linear projection on $E$ of norm one and for every $x\in E$
\[\| x-P_n(x)\|_E= \|\nu(x)-P_n(\nu(x))\|_N\to 0\]
when $n\to \infty$. 

Since the Banach lattice $N$ is order continuous, its dual space $N^*$ is also a Banach sequence space. It is well known that the dual space $N^*$ to  the $N$-direct sum $E=(\mathop{\oplus}\limits_n E_n)_N$ is then $E^*=(\mathop{\oplus}\limits_n E_n^*)_{N^*}$. If moreover $(p_n)$ is bounded away from 1,  then $N^*$ is the Nakano sequence space with the conjugate exponent sequence $(p_n^*)$, with an equivalent norm, and since the spaces $E_n$ are finite dimensional we have  that $E^{**}=(\mathop{\oplus}\limits_n E_n^{**})_{N}=(\mathop{\oplus}\limits_n E_n)_{N}=E$  (with the same norm). Thus in this case $E$ is reflexive. This remains true if a finite number of the $p_n$ equal 1 and the remainder of the exponents are bounded away from 1 (\textit{e.g.}, when the sequence $(p_n)$ converges to 2).

\begin{proposition}\label{ultrapower-N}
If the Nakano sequence space $N$ has its exponent sequence converging to 2, and if the linear spaces $E_n$ are finite dimensional, then every ultrapower of their Nakano direct sum $E=\big(\mathop{\oplus}\limits_n E_n\big)_N$  has the modular decomposition $E_\U=E\oplus_m \H$, where  $E$ is the diagonal copy of $E$ in $E_\U$ and $\H$ is the kernel of the weak limit projection $P:E_\U\to E$. Moreover
\[\Theta_\U(x+h)=\Theta(x)+\|h\|^2\]
for every $x\in E,h\in \H$.
If, moreover, the Jordan-von~Neumann constants $a(E_n)$ converge to 1, then the space $\H$ is linearly isometric to a Hilbert space, and $E=\big(\mathop{\oplus}\limits_n E_n\big)_N$ is asymptotically hilbertian.
\end{proposition}
\begin{proof} 1) {\it The modular decomposition of $E_\U$}.
As explained at the beginning of section \ref{sec:General}, the canonical image $D(E)$ of $E$ in $E_\U$ is the range of the weak limit projection $P$. Let $\H=\ker P$; we prove first that the direct sum $E_\U=E\oplus \H$ is modular.

Note that if $\xi\in \H$, then for every bounded family $(x_i)$ in $E$ representing $\xi$ and every $n\in \mathbb N$ we have
\[w-\lim_{i,\U} x_i(n)=0 \,\text{.}\]
If $x\in E$ has finite support relative to $N$, that is $n_x=\sup\{n: x(n)\ne 0\}<\infty$, and $\xi\in\H$, we can find a representing family $(x_i)$ for $\xi$ with $x_i(n)=0$ for every $n\le n_x$ and $i$. Then $\Theta(x+x_i)=\Theta(x)+\Theta(x_i)$, which shows that
\[\Theta_\U(x+\xi)= \Theta(x)+\Theta_\U(\xi) \,\text{.} \]
This equality extends to every $x\in E$, by density of finitely supported elements in $E$ and continuity of $\Theta_\U$. 

\smallskip
2) {\it $\Theta_\U(\xi)=\|\xi\|^2$ whenever $\xi\in\H$}. It suffices to prove that the modular restricted to $\H$ is 2-homogeneous.  Given $\xi\in \H$ and $n\in \mathbb N$, we can choose a family $(x_i)$ in $E$, representing $\xi$ and such that $ x_i(k)=0$ for $k=1\dots n$ and all $i\in I$. Then for every $\lambda\in \K$ and $i \in I$
\begin{align*}
 \left |\Theta(\lambda x_i)-|\lambda|^2\Theta(x_i) \right| &= \left |\sum_{k=n}^\infty \left \|\lambda x_i(k)\right\|^{p_k} - |\lambda|^2 \sum_{k=n}^\infty \left \| x_i(k) \right\|^{p_k}\right|\\
 & \le  \sum_{k=n}^\infty \left | \left |\lambda\right |^{p_k} - |\lambda|^2 \right | \left \| x_i(k) \right\|^{p_k} \le \max_{k\ge n}  \left | \left |\lambda\right |^{p_k} - |\lambda|^2 \right | \Theta(x_i) \,\text{.}
\end{align*}
It follows that
\[  \left |\Theta_\U(\lambda \xi)-|\lambda|^2\Theta_\U(\xi) \right| \le \max_{k\ge n}  \left | \left |\lambda\right |^{p_k} - |\lambda|^2 \right | \Theta_\U(\xi) \,\text{.}\]
Then since $p_n\to 2$, by letting $n\to\infty$  we obtain $\Theta_\U(\lambda \xi)=|\lambda|^2\Theta_\U(\xi)$.

\smallskip
3) {\it $\H$ is linearly isometric to a Hilbert space}.
If $x,y\in E_n$ and $p_n\le 2$ we have
\begin{align*}
{\|x+y\|^{p_n}+\|x-y\|^{p_n}\over 2} &\le \left({\|x+y\|^2+\|x-y\|^2\over 2}\right)^{p_n/2}\le a(E_n)^{p_n/2}\left(\|x\|^2+\|y\|^2\right)^{p_n/2}\\
&\le a(E_n)^{p_n/2} \left(\|x\|^{p_n}+\|y\|^{p_n}\right) \,\text{.}
\end{align*}
If $p_n\ge 2$ the inequalities are reversed
\[{\|x+y\|^{p_n}+\|x-y\|^{p_n}\over 2} \ge a(E_n)^{-p_n/2} (\|x\|^{p_n}+\|y\|^{p_n}) \,\text{.} \]
Setting $u=x+y$ and $v=x-y$ we obtain in this case
\[{\|u\|^{p_n}+\|v\|^{p_n}\over 2}\ge a(E_n)^{-p_n/2}\left(\bigg\|{u+v\over 2}\bigg\|^{p_n}+\bigg\|{u-v\over 2}\bigg\|^{p_n}\right)\]
and relabelling the variables and rearranging the preceding inequality
\[{\|x+y\|^{p_n}+\|x-y\|^{p_n}\over 2} \le 2^{p_n-2}a(E_n)^{p_n/2} \left(\|x\|^{p_n}+\|y\|^{p_n}\right) \text{.} \]
Set $\alpha_n=a(E_n)^{p_n/2}\max(1,2^{p_n-2})$; we then have $\alpha_n\to 1$ when $n\to \infty$.

Now assume that $x,y\in E$ with $x(k)=0=y(k)$ whenever $k<n$. We then have
\[{\Theta(x+y)+\Theta(x-y)\over 2}\le \beta_n (\Theta(x)+\Theta(y))\]
where $\beta_n=\sup\{\alpha_k: k\ge n\}$. Observe that $\beta_n\to 1$ when $n\to\infty$.

Passing to the ultrapower, consider $\xi,\eta\in E_\U$ that are represented by families $(x_i)$ and $(y_i)$
respectively, with $x_i(k)=0=y_i(k)$ for all $k<n$ and $i \in I$.
By the previous argument we have 
\[{\Theta_\U(\xi+\eta)+\Theta_\U(\xi-\eta)\over 2}\le \beta_n (\Theta_\U(\xi)+\Theta_\U(\eta)) \,\text{.}\]
When $\xi,\eta\in \H$, the preceding inequality is valid for all $n\in\mathbb N$, hence 
\[{\Theta_\U(\xi+\eta)+\Theta_\U(\xi-\eta)\over 2}\le \Theta_\U(\xi)+\Theta_\U(\eta) \,\text{.} \]
Since $\Theta_\U(\xi)=\|\xi\|^2$ whenever $\xi\in\H$, the Banach space $\H$ satisfies the parallelogram inequality and thus is linearly isometric to a Hilbert space.
\end{proof}

\begin{remark}
Proposition \ref{ultrapower-N} suggests that Nakano spaces like $N$ (as in Corollary \ref{cor:Nakano sum categoricity})  are ``close to being hilbertian''. However they need not be linearly isomorphic to a Hilbert space. Indeed, from a result of Nakano himself \cite{N} it is easy to deduce the following fact:
\begin{fact}\label{cond:N=ell 2}
The Nakano space $N=\ell_{(p_n)}$ has an equivalent hilbertian norm iff for some $c>0$ the series $\sum_{n=1}^\infty c^{\frac{2p_n}{|p_n-2|}}$ is convergent.
\end{fact}
Indeed by Theorem 1 in \cite{N} this condition is necessary and sufficient for the unit vector basis of $N$ to be equivalent to the $\ell_2$ basis. If $N$ has an equivalent hilbertian norm, its unit vector basis is not necessarily  an orthonormal basis for the hilbertian structure, but it remains unconditional in the hilbertian norm (since unconditionality is preserved by linear isomorphisms, although the unconditionality constant may, of course, change). Since unconditional bases in a Hilbert space are all equivalent to the $\ell_2$ unit basis, Nakano's condition must then hold true.
\end{remark}

\section{First example: Nakano direct sums}\label{sec:Nakano-direct-sums}

The main result of this section (Corollary \ref{cor:Nakano sum categoricity}) states that every Nakano direct sum  of finite dimensional normed spaces associated to a Nakano space $N$ with exponent sequence  converging to 2,  {\it but different from 2}, satisfies condition (ii) of Theorem \ref{Criterion:UC}. The isometries of general Nakano spaces with exponent function strictly greater than 2 were studied in the article \cite{JKP}. Here we can avoid the latter restriction on $(p_n)$ by taking advantage of the fact that the Banach lattice $N$ is atomic. We state first a Proposition where the condition $p_n\ne 2$ is not required, with a view to getting some partial results also in this case (Corollary \ref{cor:Nakano space categoricity} and Remark \ref{rem:general NDS}).

\begin{proposition}\label{embeddings N to N+H}
Let $H$ be any Hilbert space, and $N$ be a Nakano sequence space  with exponent sequence $p_n\to 2$. Further, let $(E_n)$ be a sequence of finite dimensional normed spaces and $E=(\mathop{\oplus}\limits_n E_n)_N$ their $N$-direct sum. Consider also the partial direct sums $E_h$ and $E_{nh}$ corresponding respectively to the set of indices  $\{n:\ p_n=2\}$, and $\{n:\ p_n\ne 2\}$.  Then every linear isometric embedding from $E$ into $E\oplus_m H$ is modular preserving and maps $E_{nh}$ onto $E_{nh}$ and $E_h$ into $E_h\oplus_m H$.
\end{proposition}
We first give two lemmas in preparation for the proof of Proposition \ref{embeddings N to N+H}.  Fix $N$, $(E_n)$, and $E=(\mathop{\oplus}\limits_n E_n)_N$ as in the hypotheses of the Proposition.  We regard $\K$ as the 1-dimensional modular space by taking its modular to be the square of the absolute value.

If a sequence $(x_n)$ in a Hilbert space $H$ is weakly null, then for every $x\in H$ we have
\[ \|x+x_n\|^2-(\|x\|^2+\|x_n\|^2)\to 0 \,\text{.}\]
In particular if $\|x_n\|\to a$ then $\|x+x_n\|^2\to \|x\|^2+a^2$. An analogous property for $E$ is given by the next lemma.

\begin{lemma}\label{weakly-null-seq}
Let $(x_n)$ be a weakly null sequence in $E$ with $\Theta(x_n)\to 1$. Then for every $x\in E$ and $t\in\K$, we have 
\[\Theta(x+tx_n)\to \Theta(x)+|t|^2 \hbox{ and } \|x+tx_n\|_E\to \|x+t\|_{E\oplus_m \K} \,\text{.} \]
\end{lemma}

\begin{proof}
Let $\U$ be a nonprincipal ultrafilter on $\mathbb N$. Since $x_n\to 0$ weakly,  the corresponding element $\xi=[x_n]_\U$ of $E_\U$ belongs to $\ker P=\H$. Thus 
\[\lim_{n,\U}\Theta(x+tx_n)= \Theta_\U(x+t\xi)=\Theta(x)+t^2\Theta_\U(\xi)=\Theta(x)+t^2 \,\text{.} \]
Since the $\U$-limit does not depend on $\U$, it is also the ordinary limit.  Similarly
\[\lim_{n,\U}\|x+tx_n\|_E= \|x+t\xi\|_{E\oplus_m \H}=\inf\{\lambda>0: \Theta\!\left(\frac x\lambda\right)+\frac {t^2}{\lambda^2}=1\} =\|x+t\|_{E\oplus_m\K} \,\text{.} \]
\end{proof}

The next lemma, valid for any modular space $M$, yields an estimation of $\|x+t\|_{M\oplus_m\K}$ for small $x\in M$.  By homogeneity of the norm, it suffices to do this for $t=1$.

\begin{lemma}\label{norm-calc}
Let $M$ be a modular space with convex modular $\Theta$ satisfying a $\Delta_2$ condition. Then for $x\to 0$ in $M$ we have
\[ \| x+1\|_{M\oplus_m\K}-1\sim \frac 12 \Theta(x) \,\text{.} \]
\end{lemma}
\begin{proof}
We have
\[ 1=\Theta\left(x+1\over \| x+1\|_m\right)=  \Theta\left( x\over \| x+1\|_m\right)+{1\over \| x+1\|_m^2}\]
hence
\[\| x+1\|_m^2-1= \| x+1\|_m^2\,\Theta\!\left( x\over \| x+1\|_m\right)\sim \Theta( x)\]
and the result follows since $\| x+1\|_m^2-1\sim 2(\| x+1\|_m-1)$.
\end{proof}

\begin{proof}[Proof of Proposition \ref{embeddings N to N+H}]
Let $T$ be an isometric embedding from $E$ into $E\oplus_m H$. We denote by $\bbar\Theta$ the modular on $E\oplus_m H$; that is, $\bbar\Theta(x+h)=\Theta(x)+\|h\|^2$ whenever $x\in E,h\in H$. Let $(\bar p_k)_{k\in \mathbb N}$ be an enumeration of the distinct values of the exponent sequence $(p_n)$, and $\bar E_k=\Span\{ e_n\otimes E_n: p_n=\bar p_k\}$.  It will be sufficient to prove that $T$ maps each $\bar E_k$ into itself, and $E_h$ into $E_h\oplus H$: indeed since $T$ is injective and the spaces $\bar E_k$ with  $\bar p_k\ne 2$ have finite dimension, so  $T$ will in fact map each of those $\bar E_k$ onto itself, and thus $E_{nh}$ onto $E_{nh}$. Moreover for $x\in \bar E_k$, $\Theta(x)=\|x\|^{\bar p_k}$, and for $x\in H$, $\Theta(x)=\|x\|^2$, so that the isometry $T$ will be modular preserving. 

Let $j\in\mathbb N$ be fixed, and for every $n\in \mathbb N$ choose $x_n\in E_n$ with $\|x_n\|=1$. Then $x_n\to 0$ weakly in $E$, and for every $x\in E$  we have by Lemma \ref{weakly-null-seq}
 \[ \|x+x_n\|_E\to \| x+1\|_{E\oplus_m \K} \,\text{.} \]
Then since $T$ is an isometry
\begin{align}\label{cv1}
\|Tx+Tx_n\|_{E\oplus_m H}= \| x+x_n\|_E\To_{n\to\infty} \|x+1\|_{E\oplus_m \K} \,\text{.}
\end{align}
On the other hand, since $x_n\to 0$ weakly in $E$ as $n\to\infty$, we see $Tx_n\to 0$ weakly in $E\oplus H$. If we let $Tx_n=u_n+v_n$, $u_n\in E$, $v_n\in H$ be the decomposition of $Te_n$ in the direct sum $E\oplus H$, then we have separate weak convergences $u_n\to 0$ in $E$ and $v_n\to 0$ in $H$.  Consider a subsequence $(n_k)$ such that $\Theta(u_{n_k})$ and $\|v_{n_k}\|^2$ converge to, say, $a^2$ and $b^2$ respectively. Note that $a^2+b^2=\lim_k \overline\Theta(Tx_{n_k})= 1$ since $\|Tx_n\|=\|x_n\|=1$. Let $Tx=u+v$, with $u\in E$ and $v\in H$. Then for every $\lambda>0$
\begin{align*}
\bbar\Theta((\lambda Tx+Tx_{n_k}))&= \Theta(\lambda (u+u_{n_k}))+\|\lambda (v+v_{n_k})\|^2 \\
&\to \Theta(\lambda u)+\lambda^2a^2+\|\lambda v\|^2+\lambda^2 b^2  \quad\hbox{ (by Lemma \ref{weakly-null-seq})}\\
&=\bbar\Theta(\lambda Tx)+\lambda^2 \,\text{.}
\end{align*}
Since the limit is independent of the subsequence,  $\bbar\Theta(\lambda (Tx+Tx_{n}))\to \bbar\Theta(\lambda Tx)+\lambda^2$ and choosing $\lambda=(\|Tx+1\|_{(E\oplus_m H)\oplus_m\K})^{-1}$ it follows that
\begin{align}\label{cv2}
\|Tx+Tx_n\|_{E\oplus_m H}\to \| Tx+1\|_{(E\oplus_m H)\oplus_m \K} \,\text{.}
\end{align}
Comparing (\ref{cv1}) and  (\ref{cv2}), we see that
\[\|x+1\|_{E\oplus_m \K}= \|Tx+1\|_{(E\oplus_m H)\oplus_m \K} \,\text{.} \]
Assume now that $x\in\bar E_j$. By Lemma \ref{norm-calc} we have for $\lambda\to 0$
\[ \|\lambda x+1\|_m-1\sim \frac 12\Theta(\lambda x)=\frac 12 |\lambda|^{\bar p_j}\|x\|^{\bar p_j}\]
and
\[ \|\lambda Tx+1\|_m -1\sim \frac 12 \bbar\Theta(\lambda Tx)\]
hence
\[\bbar\Theta(\lambda Tx)\sim |\lambda|^{\bar p_j}\|x\|^{\bar p_j} \,\text{.} \]
Assume that $\|x\|=1$. We decompose $Tx=\sum_k u_k+v$, where $u_k\in \bar E_k$ and $v\in H$. Then
\begin{align}\label{eq:Theta}
\bbar\Theta(\lambda Tx) &= \sum_k \Theta(\lambda u_k)+\|\lambda v\|_H^2 = \sum_k |\lambda|^{\bar p_k} \Theta(u_k)+|\lambda|^2 \|v\|_H^2
\end{align}
and hence
\begin{align}\label{equivalence}
1\sim |\lambda|^{-\bar p_j}\bbar\Theta(\lambda Tx) = \sum_k |\lambda|^{\bar p_k-\bar p_j} \Theta(u_k)+|\lambda|^{2-\bar p_j} \|v\|_H^2 \,\text{.}
\end{align}
This implies that $\Theta(u_k)=0$ if $\bar p_k<\bar p_j$ and $v=0$ if $2<\bar p_j$ (otherwise the right side of
equation (\ref{equivalence}) goes to $+\infty$ when $\lambda\to 0$).  On the other hand
\[\sum_{k:\,\bar p_k>p_j}|\lambda|^{\bar p_k-\bar p_j} \Theta(u_k)\to 0 \hbox{ and } |\lambda|^{2-\bar p_j} \|v\|_H^2\to 0 \hbox{ if } 2>\bar p_j \]
thus the right side of (\ref{equivalence}) is equivalent to $\Theta(u_j)$. We have thus $\Theta(u_j)=1$. 
Since $\|Tx\|=\|x\|=1$ we have $\Theta(Tx)=1$ and thus equality (\ref{eq:Theta}) with $\lambda=1$ shows that $\Theta(u_k)=0$ for all $k\ne j$, as well as $\|v\|_H=0$ if $2\ne \bar p_j$. Finally, we conclude $Tx\in \bar E_j$ as desired, except if $\bar p_j=2$, in which case $Tx\in \bar E_j\oplus H$.
\end{proof}

By Propositions \ref{ultrapower-N} and \ref{embeddings N to N+H} and Theorem \ref{Criterion:UC} we conclude:
\begin{corollary} 
\label{cor:Nakano sum categoricity}
Let $N$ be a Nakano sequence space with exponent sequence $(p_n)$, where  $p_n\ne 2$ for all $n$, and $p_n$  converges to 2.  Let $(E_n)$ be a sequence of finite dimensional normed spaces whose Jordan-von Neumann constants converge to 1.  Then the elementary class of the Nakano direct sum $E=(\oplus E_n)_N$ is equal to the class of all modular direct sums $E\oplus_m H$ of $E$ with arbitrary Hilbert spaces, and hence it is uncountably categorical.
\end{corollary}

In the scalar case we can drop the condition $p_n\ne 2$. Denote by $N_{nh}$ the Nakano space associated to the subsequence (which may be finite or not) that consists of the exponents that differ from 2 .

\begin{corollary} \label{cor:Nakano space categoricity}
Let $N$ be a Nakano sequence space with exponent sequence $(p_n)$ converging to 2.   Then the elementary class of 
$N$ consists of:
\begin{itemize}
\item[--] the class of all modular direct sums $N_{nh}\oplus_m H$ of $N_{nh}$ with an arbitrary Hilbert space $H$, if an infinity of exponents $p_n$ differ from 2.

\item[--] the class of all modular direct sums $N_{nh}\oplus_m H$ of $N_{nh}$ with an infinite dimensional Hilbert space $H$, if not.
\end{itemize}
In both cases this class is uncountably categorical.
\end{corollary}

\begin{proof}
If $N_{nh}$ is infinite dimensional, then by Corollary \ref{cor:Nakano sum categoricity}, its elementary class consists of all modular direct sums $N_{nh}\oplus_m H$ of $N_{nh}$ with an arbitrary Hilbert space $H$. In particular $N$ is elementarily equivalent to $N_{nh}$, and thus its elementary class is the same.

If $N_{nh}$ is finite dimensional, then it follows from Proposition \ref{embeddings N to N+H} that every isometric linear embedding from $N_{nh}\oplus_m \ell_2$ into $N_{nh}\oplus_m H$ sends $N_{nh}$ onto itself. Then by Theorem \ref{Criterion:UC}, 
the elementary class of $N_{nh}\oplus_m\ell_2$ consists of all  modular direct sums $N_{nh}\oplus_m H$ of $N_{nh}$ with an arbitrary infinite dimensional Hilbert space $H$. This class contains $N$, the elementary class of which is thus the same.
\end{proof}

\begin{remark}
We have a similar result to Corollary \ref{cor:Nakano space categoricity} for Nakano direct sums if we require that $E_n$ is 1-dimensional whenever $p_n=2$.
\end{remark}
\begin{remark}\label{rem:general NDS}
In section \ref{sec:2-direct-sums} we shall give several examples of 2-direct sums $F=(\mathop{\oplus}\limits_n F_n)_{\ell_2}$, where the normed spaces $F_n$ are finite dimensional, such that $F$ satisfies the hypotheses of Theorem \ref{Criterion:UC}. 
Given such an $F$, we consider further a Nakano direct sum $E=(\mathop{\oplus}\limits_n E_n)_N$, where $N$ is a Nakano sequence space with exponents $(p_n)$ all distinct from 2 and such that either $(p_n)$ is finite or it converges to 2.
Then the modular sum $E \oplus_m F$ satisfies the hypotheses of Theorem \ref{Criterion:UC}.  Indeed, if $H$ is any Hilbert space, it follows from Proposition \ref{embeddings N to N+H} that every linear isometric map from $E\oplus_m F$ into $E\oplus_m F\oplus_m H$ maps $E$ onto $E$ and $F$ into $F\oplus_m H$, and thus  $F$ onto $F$. 
\end{remark}

In the spirit of the preceding Remark we now give a general lemma about isometries of modular direct sums, which will have several applications in the next section.
\begin{lemma}\label{isom-modular-sum}
Let $E_1, F_1,E_2,F_2$ be modular spaces, and $T:E_1\oplus_m F_1\to E_2\oplus_m F_2$ a linear isometric embedding. If $T(E_1)=E_2$ then $T(F_1)\subset F_2$.
\end{lemma}
\begin{proof}
Let $f\in F_1$ with $\|f\|=1$ and decompose $Tf=x+g$ with $x\in E_2$ and $g\in F_2$.  We have
\begin{align}\label{eq}
1=\overline\Theta(Tf)=\Theta(x)+\Theta(g) \,\text{.}
\end{align}
Since $T \colon E_1\to E_2$ is onto we can find $y\in E_1$ with $x=Ty$, and thus $g= T(f-y)$. Then
\[ \|g\|=\|T(f-y)\|=\|f-y\|\ge \|f\|=1 \]
because in the modular sum, factor projections are contractive.  Using (\ref{eq}) we find that $\Theta(x)=0$ and thus $x=0$.
\end{proof}

\section{Second example: 2-direct sums}\label{sec:2-direct-sums}

In this section we consider direct sums of the form $E=(\mathop{\oplus}\limits_n E_n)_{\ell_2}$, where the normed spaces $E_n$ are finite dimensional. The norm of $x=(x_n)$ is given by $\|x\|^2= \sum_{n=1}^\infty \|x_n\|_{E_n}^2$. Since this is a special case of Nakano direct sum, with constant exponent function $p_n\equiv 2$, Proposition \ref{ultrapower-N} of Section \ref{sec:asympt-H-Nakano} applies:  if $a_n(E_n)\to 1$ then $E$ is asymptotically hilbertian; in fact $E_\U=E\oplus_2 \H$ for some Hilbert space $\H$ depending on $\U$. However in contrast to the Nakano direct sums  treated in Section \ref{sec:Nakano-direct-sums}, here we need some relatively strong hypotheses on the $E_n$'s for proving that the isometric embeddings from $E$ into $E\oplus_2\H$ send each $E_n$ into itself.  See Propositions \ref{prop:rigidity} and \ref{prop:rigidity:2} and Corollary \ref{cor:rigid:mix} for the kinds of assumptions under which we have been able to carry out the required arguments.  Example \ref{examples} summarizes the specific examples $E=(\mathop{\oplus}\limits_n E_n)_{\ell_2}$ for which we prove uncountable categoricity in this section.

\begin{remark}

If the Banach-Mazur distance $d(E_n,\ell_2^{d_n})$ is not bounded, where $d_n=\mathrm{dim}\,E_n$, then the space $E$ is not linear-topologically isomorphic to a Hilbert space.
Indeed if $E$ is C-linearly isomorphic to a Hilbert space, so is every closed linear subspace of $E$.
\end{remark} 

\begin{proposition}\label{prop:rigidity}
Let $(E_n)$ be a sequence of finite dimensional Banach spaces. 
Assume that for some sequence of exponents  $p_n> 2$, with $p_n\to 2$,  the following conditions are satisfied:

\noindent a) For every $n$, and every $x,y\in E_n$ we have 
\[\|x+y\|^2_{E_n}+ \|x-y\|_{E_n}^2
\ge 2(\|x\|_{E_n}^{p_n}+\|y\|_{E_n}^{p_n})^{2/p_n} \,\text{;} \]

\smallskip
\noindent b) For every $n$ there exists a basis $\mathcal B_n$ of $E_n$ such that for every $y\in \mathcal B_n$ there is $x\in \mathcal B_n$ such that $(x,y)$ is an $\ell_{p_n}$-pair, that is $\|x+\lambda y\|_{E_n}^{p_n}=1+|\lambda|^{p_n}$ for every $\lambda\in\reel$.
\smallskip

Then the hypotheses of Theorem \ref{Criterion:UC} are satisfied, and the Banach space $E=(\mathop{\oplus}\limits_n E_n)_{\ell_2}$ is uncountably categorical.
\end{proposition}
\begin{proof}
1) By the hypothesis (a) and H\"older's inequality we have for every $x,y\in E_n$
\[\|x+y\|^2_{E_n}+ \|x-y\|_{E_n}^2\ge 2\times 2^{-2/r_n}(\|x\|_{E_n}^2+\|y\|_{E_n}^2)\]
with $\frac 1{r_n}=\frac 12-\frac 1{p_n}$. Thus by (\ref{eq:a(X)}),  
$a(E_n)\le 2^{2/r_n}\to 1$ as $n\to \infty$, and it follows from Proposition \ref{ultrapower-N} that every ultrapower of $E$ is of the form $E\oplus_2 \H$, where $\H$ is a Hilbert space.

2) Let $S: E_n\to E\oplus_2 H$ be a linear isometric embedding, we show that its range is included in $E$ (this will require only condition (b)). This will easily imply that any linear isometric embedding $T: E\to E\oplus_2 H$ sends $E$ into $E$.  Let $\mathcal B_n$ be a basis of $E_n$ as in the condition (b) of the proposition, it is sufficient to prove that $Sy\in E$ for every $y\in\mathcal B_n$. Let $x\in\mathcal B_n$ be choosen such that $(x,y)$ is an $\ell_{p_n}$-pair. Let $u=Sx$, $v=Sy$ and $u=\sum_k u(k)+u_H$, $v=\sum_k v(k)+v_H$, with $u(k),v(k)\in E_k$ and $u_H,v_H\in H$. Then 
\begin{align}\label{eq:1}
	\begin{split}
{\|u+\lambda v\|^2+ \|u-\lambda v\|^2\over 2}&= \sum_{k=1}^\infty {\|u(k)+\lambda v(k)\|^2+ \|u(k)-\lambda v(k)\|^2\over 2} \cr
 & \phantom{\|u(k)+\lambda v(k)\|^2)\|}\hskip-0cm 
+{\|u_H+\lambda v_H\|^2+ \|u_H-\lambda v_H\|^2\over 2} \,\text{.}
	\end{split}
 \end{align}
 Since $\|u\pm \lambda v\|^{p_n}= \|x\pm \lambda y\|^{p_n}= 1+|\lambda|^{p_n}$, the left side of equation (\ref{eq:1}) is equal to $( 1+|\lambda|^{p_n})^{2/p_n}$. On the other hand, by convexity of $\|\cdot\|^2$ and the parallelogram identity in $H$,  the right side of (\ref{eq:1}) is bigger than
 \[\sum_{k=1}^\infty  \|u(k)\|^2+(\|u_H\|^2+|\lambda|^2\|v_H\|^2)= \|u\|^2+|\lambda|^2\|v_H\|^2=1+|\lambda|^2\|v_H\|^2 \,\text{.}
 \]
Thus, since $p_n> 2$, we have
 \[\|v_H\|^2\le  {(1+|\lambda|^{p_n})^{2/p_n}-1\over \lambda^2}\le \frac2{p_n}|\lambda|^{p_n-2}\To_{\lambda\to 0} 0 \,\text{;}
 \]
hence $v_H$, the $H$-component of $v$, is $0$. 
 
3) Now assuming both conditions (a) and (b) we show that the range of any isometric embedding $S: E_n\to E\oplus_2 H$ is included in $\mathop{\oplus}\limits_{p_m\ge p_n} \!\!E_m$. We keep the notation of the preceding part. For every $m\ge 1$, the right side of equation (\ref{eq:1}) is by condition (a) greater than
 \[ \sum_{k\ne m} \|u(k)\|^2+ (\|u(m)\|^{p_m}+|\lambda|^{p_m}\|v(m)\|^{p_m})^{2/p_m}+\|u_H\|^2 \,\text{.}
 \]
Hence 
 \[( 1+|\lambda|^{p_n})^{2/p_n}\ge 1+ (\|u(m)\|^{p_m}+|\lambda|^{p_m}\|v(m)\|^{p_m})^{2/p_m} - \|u(m)\|^2 \]
for every $\lambda\in\mathbb R$. If $u(m)=0$ we get
\[( 1+|\lambda|^{p_n})^{2/p_n}\ge 1+ |\lambda|^2 \|v(m)\|^2\]
and deduce $v(m)=0$ in the same way we did for $v_H$. If $u(m)\ne 0$ we get
\begin{align*}
( 1+|\lambda|^{p_n})^{2/p_n}-1&\ge \|u(m)\|^2\big( (1+|\lambda|\|v(m)\|/\|u(m)\|)^{p_m}-1\big)^{2/p_m}\\ 
&\sim \frac 2{p_m} \|u(m)\|^{2-p_m}\|v(m)\|^{p_m}|\lambda|^{p_m}\ \hbox{ as } \lambda\to 0 \,\text{;}
\end{align*}
thus if $p_m<p_n$ we get
\[ \|v(m)\|^{p_m} \lesssim \frac{p_m}{p_n}\|u(m)\|^{p_m-2}|\lambda|^{p_n-p_m}\To_{\lambda\to 0} 0 \,\text{.}
\]
Hence $v(m)$, the $E_m$-component of $v$, must vanish. Finally $Sy\in \mathop{\oplus}\limits_{p_m\ge p_n} \!\!E_m$ as was claimed.

4) Now let $T: E\to E\oplus_2 H$ be a linear isometric embedding. Let us denote by $(\bar p_k)$ an enumeration of the distinct values of the $p_n$'s (for fixing ideas we may assume the sequence $(\bar p_k)$ to be strictly decreasing). Note that since $p_n>2$ and $p_n\to 2$, each set $A_k=\{n: p_n=\bar p_k\}$ is finite.  For every $k\ge 1$ set $G_k= \mathop{\oplus}\limits_{p_n\ge \bar p_k} E_n$. By part (3) above, we have that $T(G_k)\subset G_k$. Since $G_k$ is finite dimensional and $T$ is isometric it follows that $T(G_k)=G_k$. Hence the range of $T$ contains $\bigcup_k G_k$, a dense subspace of $E$, and since this range is closed it contains $E$.
 \end{proof}
\begin{remark}\label{remark:1}
We have in fact the more precise result that $T(\bar E_k)= \bar E_k$ for every $k\ge 1$, where  $\bar E_k= \mathop{\oplus}\limits_{p_n= \bar p_k} E_n$.  For $k=1$ we have $\bar E_1=G_1$ and thus $T(\bar E_1)=\bar E_1$. For $k\ge 2$ it will be sufficient to prove that $T(\bar E_k)\subset \bar E_k$. This is done inductively using $G_k= \bar E_k\oplus_m G_{k-1}$ and Lemma \ref{isom-modular-sum}.
\end{remark}
\begin{example}
$E=\big(\oplus \ell_{p_n}^{d_n}\big)_2$ with $p_n>2$, $p_n\to 2$ and $d_n\ge 2$ satisfies the hypotheses of Proposition \ref{prop:rigidity}.
Condition (b) is clearly satisfied. As for condition (a), we have for $x,y\in\ell_p$, $p\ge 2$:
\begin{align*}
\frac {\|x+y\|^2_p+\|x-y\|^2_p}2 &= \frac {\||x+y|^2\|_{p/2}+\||x-y|^2\|_{p/2}}2 \\
&\ge \bigg\|\frac{|x+y|^2+|x-y|^2}2 \bigg\|_{p/2} \ \hbox{ (by convexity since } p/2\ge 1)\\
&= \||x|^2+|y|^2\|_{p/2}= \||x^2+y^2|^{1/2}\|_p^2\\ 
&\ge \|(|x|^p+|y|^p)^{1/p}\|_p^2  \hskip 1.3cm \hbox{ (since } p\ge 2)\\
&= (\|x\|_p^p+\|y\|_p^p)^{2/p} \,\text{.}
\end{align*}
Thus  $E$ is uncountably categorical. On the other hand if $d(E_n,\ell_2^{d_n})=d_n^{\frac 12-\frac 1{p_n}}\to\infty$ then $E$ is not linear-topologically isomorphic to a Hilbert space.
\end{example}
\begin{example}
$E=\big(\oplus S_{p_n}^{d_n}\big)_2$ with $p_n>2$, $p_n\to 2$, and $d_n\ge 2$, where $S_p^d$ is the Schatten class of exponent $p$ and dimension $d^2$ (consisting of $d\times d$ matrices with complex coefficients).

Let $\mathcal B_n$ be the basis of $S_{p_n}^{d_n}$ consisting of the matrix units $(e_{i,j})_{1\le i,j\le d_n}$. For each matrix unit $e_{i,j}$ consider another matrix unit $e_{k,\ell}$ with $i\ne k, j\ne \ell$. Then the pair $(e_{ij},e_{k\ell})$ is a $\ell_{p_n}$-pair in $S_{p_n}^{d_n}$, and the condition (b) in Proposition \ref{prop:rigidity} is satisfied. As for condition (a) we reason by interpolation. Indeed condition (a) means exactly that for $p=p_n$ the inverse of the operator $M: (x,y)\mapsto (\frac{x+y}{\sqrt 2}, \frac{x-y}{\sqrt 2})$ is contractive from $\ell_2^2(S_p^d)$ to $\ell_p^2(S_{p}^d)$. Since these spaces are complex interpolation spaces, more precisely $\ell_2^2(S_p^d)=(\ell_2^2(S_2^d),\ell_2^2(S_\infty^d))_\theta$ and $\ell_p^2(S_{p}^d)=(\ell_2^2(S_2^d),\ell_\infty^2(S_\infty^d))_\theta$ for $\theta=1-\frac 2p$, it is sufficient to verify contractivity in the cases $p=2$ and $p=\infty$. For $p=2$, $S_p^d=S_2^d$ is a Hilbert space and condition (b) follows from the parallelogram identity. For $p=\infty$, $S_p^d=M_d(\mathbb C)$ with the matrix norm (which we denote by $\|\cdot\|_\infty$) and we have for $x,y\in M_d$
\begin{align*}
\frac {\|x+y\|^2_\infty+\|x-y\|^2_\infty}2 &= \frac {\|(x+y)^*(x+y)\|_\infty+\|(x-y)^*(x-y)\|_\infty}2 \\
&\ge \bigg\|\frac{(x+y)^*(x+y)+(x-y)^*(x-y)}2 \bigg\|_\infty \\
&= \|x^*x+y^*y\|_\infty\\ 
&\ge \max(\|x^*x\|_\infty, \|y^*y\|_\infty) = \big(\max(\|x\|_\infty, \|y\|_\infty)\big)^2 \,\text{.}
\end{align*}
Thus here again  $E$ is uncountably categorical. Note that if $d_n^{\frac 12-\frac 1{p_n}}\to\infty$, it follows from \cite[Theorem 2.1]{P}
that E is not linear-topologically embeddable in any space with local unconditional structure with nontrivial cotype (in particular, any Banach lattice with nontrivial concavity).
\end{example}

Next we present another criterion for uncountable categoricity, similar to Proposition \ref{prop:rigidity}, but with exponents strictly less than 2.

\begin{proposition}\label{prop:rigidity:2}
Assume that for some sequence of exponents  $1\le p_n<2$, with $p_n\to 2$,  the following conditions are satisfied:

\noindent a) For every $n$, and every $x,y\in E_n$ we have
\[\|x+y\|^2_{E_n}+ \|x-y\|_{E_n}^2\le 2(\|x\|_{E_n}^{p_n}+\|y\|_{E_n}^{p_n})^{2/p_n} \,\text{.}\]
b) For every $n$ there exists a basis $\mathcal B_n$ of $E_n$ such that for every $y\in \mathcal B_n$ there is $x\in \mathcal B_n$ such that $(x,y)$ is an $\ell_{p_n}$-pair, that is $\|x+\lambda y\|_{E_n}^{p_n}=1+|\lambda|^{p_n}$ for every $\lambda\in\reel$.
\smallskip

Then the hypotheses of Theorem \ref{Criterion:UC} are satisfied, and the Banach space $E=(\mathop{\oplus}\limits_n E_n)_{\ell_2}$ is uncountably categorical.
\end{proposition}
\begin{proof}
1) Hypothesis (a) implies that for every $x,y\in E_n$
\[\|x+y\|^2_{E_n}+ \|x-y\|_{E_n}^2\le 2\times 2^{2/r_n}(\|x\|_{E_n}^2+\|y\|_{E_n}^2)\]
where $\frac 1{r_n}=\frac 1{p_n}-\frac 12$, and therefore $a_n(E_n)\le 2^{2/r_n}$ by (\ref{def:a(X)}).

2) We prove that if $S:\bar E_1\to E\oplus_2 H$ is an isometric embedding then $S\bar E_1= \bar E_1$. (As before, $\bar E_k=\big(\mathop{\oplus}\limits_{p_n=\bar p_k} E_n\big)_2$, where $\bar p_1<\bar p_2 <\dots$ is the sequence of distinct values of the $p_n$'s rearranged now in increasing order).

Let $(x,y)$ be an $\ell_{\bar p_1}$-pair in $\bar E_1$; we claim that the $H$-component as well as the $E_n$-components for $p_n>\bar p_1$ of $v=Sy$ all vanish. Let $u=Sx$, $v=Sy$; then for every $\lambda\in \reel$
\begin{align}
\notag(1+|\lambda|^{\bar p_1})^{2/\bar p_1}&= \frac{\|u+\lambda v\|^2+\|u-\lambda v\|^2}2\\  \notag
&\le \sum_{n=1}^\infty \big(\|u(n)\|^{\bar p_n}+\|\lambda v(n)\|^{\bar p_n}\big)^{2/\bar p_n} + \|u_H\|^2+\|\lambda v_H\|^2\\ 
&\le  \sum_{p_n=\bar p_1}\big(\|u(n)\|^{\bar p_1}+\|\lambda v(n)\|^{\bar p_1}\big)^{2/\bar p_1} + 
\sum_{p_n\ge\bar p_2}\big(\|u(n)\|^{\bar p_2}+\|\lambda v(n)\|^{\bar p_2}\big)^{2/\bar p_2}  \label{ineq}\\
&\phantom{le  \sum_{p_n=\bar p_1}\big(\|u(n)\|^{\bar p_1}+\|\lambda v(n)\|^{\bar p_1}\big)^{2/\bar p_1} }+\big(\|u_H\|^{\bar p_2}+\|\lambda v_H\|^{\bar p_2}\big)^{2/\bar p_2} \,\text{.} \notag
\end{align}
Let $u_{\bar E_1}=\sum\limits_{p_n=\bar p_1} u(n)$, $v_{\bar E_1}=\sum\limits_{p_n=\bar p_1} v(n)$ be the components of $u,v$  in $\bar E_1$; then by the reverse Minkowski inequality in $\ell_{\bar p_1/2}$ (note that $\bar p_1/2\le 1$) we have
\begin{align*}
 \sum_{p_n=\bar p_1}\big(\|u(n)\|^{\bar p_1}+\|\lambda v(n)\|^{\bar p_1}\big)^{2/\bar p_1} &\le \bigg(
 \big(\sum_{p_n=\bar p_1}\|u(n)\|^2\big)^{\bar p_1/2}+\big(\sum_{p_n=\bar p_1}\|\lambda v(n)\|^2\big)^{\bar p_1/2}\bigg)^{2/\bar p_1}\\
 &= \big(\|u_{\bar E_1}\|^{\bar p_1}+\|\lambda v_{\bar E_1}\|^{\bar p_1}\big)^{2/\bar p_1}  \,\text{.}
\end{align*}
Set $G_1=\big(\mathop{\oplus}\limits_{k\ge 2}\bar E_k\oplus H\big)_2$, and let $u_{G_1}, v_{G_1}$ be the components of $u,v$ in $G_1$; treating similarly the last two terms in (\ref{ineq}), we obtain
\begin{equation}\label{ineq:2}
(1+|\lambda|^{\bar p_1})^{2/p_1}\le \big(\|u_{\bar E_1}\|^{\bar p_1}+\|\lambda v_{\bar E_1}\|^{\bar p_1}\big)^{2/\bar p_1} + \big(\|u_{G_1}\|^{\bar p_2}+\|\lambda v_{G_1}\|^{\bar p_2}\big)^{2/\bar p_2} \,\text{.}
\end{equation}
The  left side of inequality (\ref{ineq:2}) is 
\[ 1+\frac 2{p_1}|\lambda|^{\bar p_1}+o(|\lambda|^{\bar p_1})\]
while the right side of (\ref{ineq:2}) is
\begin{align*}
& \|u_{\bar E_1}\|^2+\frac 2{\bar p_1} ||u_{\bar E_1}\|^{2-\bar p_1} \|v_{\bar E_1}\|^{\bar p_1}|\lambda|^{p_1}+o(|\lambda|^{\bar p_1})
 +\|u_{G_1}\|^2+\frac 2{\bar p_2} ||u_{G_1}\|^{2-\bar p_2} \|v_{G_1}\|^{\bar p_2}|\lambda|^{\bar p_2}+o(|\lambda|^{\bar p_2})
  \\ &= 1+\frac 2{\bar p_1} ||u_{\bar E_1}\|^{2-\bar p_1} \|v_{\bar E_1}\|^{\bar p_1}|\lambda|^{\bar p_1}+o(|\lambda|^{\bar p_1}) \,\text{.}
\end{align*}
Comparing the leading terms of both sides of inequality (\ref{ineq:2}) we obtain $1\le ||u_{\bar E_1}\|^{2-\bar p_1} \|v_{\bar E_1}\|^{\bar p_1}$; since $||u_{\bar E_1}\|\le \|u\|=1$, $\|v_{\bar E_1}\|\le\|v\|=1$ this implies $\|u_{\bar E_1}\|=\|v_{\bar E_1}\|=1$. Then $ ||u_{G_1}\|^2=1-||u_{\bar E_1}\|^2=0$ and similarly $||v_{G_1}\|^2=0$, and $u,v\in\bar E_1$ as was claimed. Thus $S(\bar E_1)\subset \bar E_1$, and in fact $S(\bar E_1)=\bar E_1$ since the dimension is finite.

3) If now $T:E \to E\oplus H$ is an isometric embedding, then by part (2) we have $T(\bar E_1)=\bar E_1$. It follows that also $T(G_1)\subset G_1$ by  Lemma \ref{isom-modular-sum}. Now starting with $G_1$ in place of $E\oplus H$, the reasoning of part (2) shows that $T(\bar E_2)=\bar E_2$, etc.
\end{proof}

\begin{examples}
$E=\big(\oplus \ell_{p_n}^{d_n}\big)_2$ and $E=\big(\oplus S_{p_n}^{d_n}\big)_2$ with $1\le p_n<2$, $p_n\to 2$, and $d_n\ge 2$ satisfy the hypotheses of Proposition \ref{prop:rigidity:2}.
\end{examples}
The proof that these examples satisfy condition (a) of Proposition \ref{prop:rigidity:2} is by duality ($a(E_n)=a(E_n^*)$). 

\medskip
In certain cases we can mix the examples of Propositions \ref{prop:rigidity} and \ref{prop:rigidity:2}. 

\begin{corollary}\label{cor:rigid:mix}
 Let $E=(\mathop{\oplus}\limits_n E_n)_2$ and  $F=(\mathop{\oplus}\limits_n F_n)_2$ be two direct sums satisfying respectively the hypotheses of Propositions \ref{prop:rigidity:2} and \ref{prop:rigidity} with respective exponent sequences $1\le p_n<2$ and  $2<q_n<\infty$. Assume moreover that for some constant $C$ we have
\begin{equation}\label{eq:2-smooth}
\forall x,y\in F_n,\ \|x+y\|^2+\|x-y\|^2\le 2(\|x\|^2+\|Cy\|^2)
\end{equation}
(that is, the spaces $F_n$ are uniformly 2-uniformly smooth in the sense of \cite{BCL}).
 
 Then every linear isometric embedding $T$ of $E\oplus_2 F$ into $E\oplus_2 F\oplus_2 H$, where $H$ is any Hilbert space, maps $E$ onto $E$ and $F$ onto $F$. 
In particular the hypotheses of Theorem \ref{Criterion:UC} are satisfied by $E\oplus_2 F$, and hence that space is uncountably categorical.
\end{corollary}

\begin{proof}
 The proof that $T$ maps $E$ onto itself is the same as in Proposition \ref{prop:rigidity:2} except that we have to replace the Hilbert space $H$ by the direct sum $G= H\oplus F$. It follows from Lemma \ref{isom-modular-sum} that $T$ maps $G$ into $G$. Then apply Proposition \ref{prop:rigidity} to the restriction of $T$ to $G$.
 \end{proof}

\begin{examples}\label{examples}
$\big(\oplus \ell_{p_n}^{d_n}\big)_2$ and $\big(\oplus S_{p_n}^{d_n}\big)_2$ with $1\le p_n<\infty$, $p_n\ne 2$, $p_n\to 2$, and $d_n\ge 2$ satisfy the hypotheses of Corollary \ref{cor:rigid:mix}.
\end{examples}

\begin{proof} We recall a proof of (\ref{eq:2-smooth}) in the case of $L_p$-spaces, $p\ge 2$.
For any scalars $x,y$,  by an  inequality of Beckner (see \cite[1.e.14]{LT} for the real case; the complex case is a special case of \cite[1.e.15]{LT}) we have:
\[ {|x+y|^p+|x-y|^p\over 2}\le \left({|x+C_p\, y|^2+|x-C_p\, y|^2\over 2}\right)^{p/2} \]
with $C_p=\sqrt{p-1}$. 
We deduce when $x,y\in \ell_p^d$
\begin{align*}
{||x+y||_p^p+||x-y||_p^p\over 2}&\le \left\|\left({|x+C_p\,y|^2+|x-C_p\,y|^2\over 2}\right)\right\|_{p/2}^{p/2}   
 = \left\| |x|^2+ (p-1)|y|^2\right\|_{p/2}^{p/2} \\
&\le \left(\| |x|^2\|_{p/2}+(p-1)\| |y|^2\|_{p/2}\right)^{p/2} \hbox{(triangular inequality in $\ell_{p/2}$)}\\
&= \left(\| x\|_p^2+(p-1)\| y\|_{p}^2\right)^{p/2} \text{.}
\end{align*}
Then by concavity of the function $f(t)=t^{2/p}$:
\[ {||x+y||_p^2+||x-y||_p^2\over 2}\le  \left({||x+y||_p^p+||x-y||_p^p\over 2}\right)^{2/p}\le \| x\|_p^2+(p-1)\| y\|_{p}^2\,\,\text{.} \]
For the Schatten class case use \cite[Th. 1]{BCL} and the preceding concavity argument.
\end{proof}

\begin{remark}
In examples \ref{examples} we can drop the condition $d_n\ge 2$ and show that these spaces are uncountably categorical by  reasoning similarly as in the proof of Corollary \ref{cor:Nakano space categoricity}.
\end{remark}

\section{Addendum: finite dimensional perturbations of Hilbert spaces}\label{sec:E+H}

As a footnote to Section \ref{sec:2-direct-sums}, we prove here that for any finite dimensional normed space $E$, the direct sum $E\oplus_2 \ell_2$ is $\kappa$-categorical for every infinite cardinal number $\kappa$.  This is a relatively simple example of categoricity and the proof is reasonably short; we present it here for completeness.  For a similar but partial result for finite dimensional modular spaces and the modular direct sum, see Corollary \ref{cor:Nakano space categoricity} and the remark following it.

\begin{definition}
 A  linear projection $P$ in a Banach space $X$ is called a \emph{$2$-projection} if
 \[\forall x\in X \quad  \|x\|^2=\|Px\|^2+\|(I-P)x\|^2 \,\text{.} \]
 A closed linear subspace $E$ in $X$ is called a \emph{$2$-summand} if it is the range of a 2-projection.
\end{definition}
In other words $E$ is a 2-summand iff $X=E\oplus_2 F$ for some closed linear subspace $F$ of $X$.

\begin{theorem}\label{thm:fin dim 2-sum with Hilbert}
Let $E_0$ is a finite dimensional normed space. The following assertions are equivalent:

i) $E_0$ has no  one dimensional 2-summand;

ii)  Any linear isometric embedding $T$ of $E_0$ into the 2-direct sum $E_0\oplus_2 H$ of $E_0$ with some Hilbert space $H$ maps $E_0$ onto $E_0$.
\end{theorem}

To prepare for the proof of Theorem \ref{thm:fin dim 2-sum with Hilbert} we prove the following lemma:

\begin{lemma}\label{lem:intersection}
 Let $E_0$ be a Banach space without one dimensional 2-summand, and $T$ be an isometric linear embedding of $E_0$ into a 2-direct sum $E_0\oplus_2 H$ of $E_0$ with a Hilbert space $K$. Then $T(E_0)\cap H=(0)$.
\end{lemma}

\begin{proof}
It suffices to prove that if $Y$ is a linear subspace of $X:=E_0\oplus_2 H$ then any vector $\xi\in Y\cap H$ generates a 2-summand in $Y$. Note that since $H$ is a Hilbert space, every closed subspace is a 2-summand in $H$, in particular $H=\K\xi\oplus_2 \xi^\perp$. Then $X=\K\xi\oplus_2(E_0\oplus_2 \xi^\perp)$, and $\K\xi$ is a 2-summand in $X$. Let $P: X\to X$ be the corresponding 2-projection in $X$ with range $\K\xi$, then its restriction $P|_Y$ is a 2-projection in $Y$ with range $\K\xi$.
\end{proof}

\begin{proof}[Proof of Theorem \ref{thm:fin dim 2-sum with Hilbert}]  $ii) \Longrightarrow i)$: 
It is clear that if $E_0$ has a 2-direct decomposition $E_0=E_1\oplus_2 \K\xi_0$, with $\xi_0\ne 0$, one can define an isometric embedding $T: E_0\to E_0\oplus_2 H$ by defining $T$ as the identity on $E_1$ and $T\xi_0=\xi_1\in H$, with $\|\xi_1\|=\|\xi_0\|$. 

Let us prove now the implication $i) \Longrightarrow ii)$.
Let $T: E_0\to X:=E_0\oplus_2 H$ be an isometric embedding and $P:X\to E_0$ be the 2-projection on $E_0$ with kernel $H$. Then $PT: E_0\to E_0$ is a linear isomorphism, since $PTx=0$ is equivalent to $Tx\in H$, and $T(E_0)\cap H=(0)$ by Lemma \ref{lem:intersection}. Since $E_0$ is finite dimensional, $PT:E_0\to E_0$ is onto. Let $Q = \mbox{Id}_X - P:X \to X$, which is the 2-projection on $H$ with kernel $E_0$.  We then have:

\begin{align*}
 \|x\|^2&=\|PTx\|^2+\|QTx\|^2\\
 &= \|(PT)^2 x\|^2+\|QTPTx\|^2+\|QTx\|^2\\
 &=\dots\\
 &= \|(PT)^n x\|^2+\sum_{k=0}^{n-1}\|QT(PT)^k x\|^2 \,\text{.}
\end{align*}
Equivalently, we have a sequence of isometric linear embeddings $T_n$ of $E_0$ into $E_0\oplus_2 H_n$, where $H_n=\ell_2^n(H)$ is the 2-sum of $n$ copies of $H$, defined by
\[T_nx= ((PT)^n x, QTx,QTPTx,\dots,QT(PT)^{n-1}x) \,\text{.} \]
 All these Hilbert spaces $H_n$ can be isometrically embedded in the infinite 2-sum $H_\infty=\ell_2(H)$. For defining a limit embedding of $E_0$ into $E_0\oplus_2 H$, fix a free ultrafilter $\U$ on $\mathbb N$ and set
 \[T^\U_\infty x=(S^\U x, QTx,QTPTx,\dots, QT(PT)^nx,\dots)\]
 where 
 \[S^\U x=\lim\limits_{n,\U}\,(PT)^n x \,\text{.} \]
 This ultrafilter limit is well defined because $PT$ is a contraction and $E_0$ is finite dimensional. Then $S^\U:E_0\to E_0$ is a linear contraction.  
 Note that although $S^\U x$ depends a priori on the ultrafilter $\U$, the norm $\|S^\U x\|$ does not. In fact the sequence $(\|(PT)^nx\|)$ is non-increasing, and thus convergent, so that
 \[\|S^\U x\|=\lim_{n,\U}\|(PT)^n x\|=\lim_{n\to\infty}\|(PT)^n x\| \,\text{.} \]
 If $P_\infty$ denotes the 2-projection $E_0\oplus_2 H_\infty\to E_0$ with kernel $H_\infty$, we have clearly $S^\U=P_\infty T^\U$. Since $T^\U$ is a linear isometry it results again that $S^\U$ is a linear isomorphism of $E_0$ onto $E_0$. This map is contractive; let us show that it is in fact an isometry. We have
 \[(S^\U)^2x= S^\U\lim_{n,\U} (PT)^n x= \lim_{n,\U} S^\U (PT)^n x = \lim_{n,\U}\lim_{m,\U} (PT)^m(PT)^n x= \lim_{n,\U}\lim_{m,\U} (PT)^{m+n} x\]
 hence
 \[ \|S^\U S^\U x\|= \lim_{n,\U}\lim_{m,\U}\| (PT)^{m+n} x\| = \lim_{k\to\infty} \|(PT)^k x\|= \|S^\U x\| \,\text{.} \]
 Since $S^\U: E_0\to E_0$ is surjective, it follows that $\|S^\U y\|=\|y\|$ for every $y\in E_0$.  Then
 \[\forall x\in E_0,\quad \|x\|=\|S^\U x\|\le \|PTx\|\le \|x\|Ê\]
 and it follows that, for each $x\in E_0$, $\|PT x\| =\|Tx\|$ and thus $Tx\in E_0$.
 \end{proof}
 
\begin{corollary}
If $E$ is a finite dimensional Banach space, then the elementary class of $E \oplus_2 \ell_2$ consists exactly of all spaces $E \oplus_2 H$, where $H$ is any infinite dimensional Hilbert space.  Therefore, the elementary class of $E \oplus_2 \ell_2$ is totally categorical (\textit{i.e.}, it is $\kappa$-categorical for every infinite cardinal number $\kappa$).
\end{corollary}
\begin{proof}
Let $K$ be a hilbertian subspace of $E$ of largest dimension such that $K$ is a 2-summand of $E$, and let $E_0$ be a subspace of $E$ for which $E = E_0 \oplus_2 K$.  Evidently $E_0$ has no one dimensional 2-summand.  Regarding $E_0$ as a modular space with the modular $\Phi(x):= \|x\|^2$ and applying Theorem \ref{thm:fin dim 2-sum with Hilbert}, we see that condition (ii) in Theorem \ref{Criterion:UC}  is satisfied by $E_0$ (and a fortiori condition (ii')).  Hence the elementary class of $E_0 \oplus_2 \ell_2$ consists exactly of all spaces $E_0 \oplus_2 H$, where $H$ is any infinite dimensional Hilbert space.  The proof is completed by noting that for any infinite dimensional Hilbert space $H$, the spaces $E \oplus_2 H$ and $E_0 \oplus_2 H$ are linearly isometric.  
\end{proof}

\end{document}